\documentclass[12pt,oneside]{amsart}
\usepackage{latexsym,amssymb,amsmath}

\numberwithin{equation}{section}
\newtheorem{theorem}{Theorem}[section]
\newtheorem{lemma}[theorem]{Lemma}

\newtheorem{corollary}[theorem]{Corollary}

\theoremstyle{definition}
\newtheorem{definition}[theorem]{Definition}
\newtheorem{remark}[theorem]{Remark}
\newtheorem{example}[theorem]{Example}

\newtheorem{convention}[theorem]{Convention}

\numberwithin{equation}{section}
\newcommand{\beq}{\begin{equation}}
\newcommand{\eeq}{\end{equation}}

\newcommand{\pr}{\mathbb{P}}

\newcommand{\supp}{\operatorname{Supp}}

\newcommand{\rk}{\operatorname{rk}}

\newcommand{\N}{\mathbb{N}}

\newcommand{\M}{\mathcal{M}}

\newcommand{\f}{\footnotesize}
\newcommand{\dd}{\delta}

\title{Separators of fat points in $\pr^n$}
\thanks{Updated: Revised Version}

\begin{document}
\author{Elena Guardo}
\address{Dipartimento di Matematica e Informatica\\
Viale A. Doria, 6 - 95100 - Catania, Italy}
\email{guardo@dmi.unict.it}

\author{Lucia Marino}
\address{Dipartimento di Matematica e Informatica\\
Viale A. Doria, 6 - 95100 - Catania, Italy}
\email{lmarino@dmi.unict.it}

\author{Adam Van Tuyl}
\address{Department of Mathematical Sciences\\
Lakehead University, Thunder Bay, ON, P7B 5E1, Canada}
\email{avantuyl@lakeheadu.ca}

\keywords{fat points, Hilbert function, resolutions, separators}
\subjclass{13D40, 13D02, 14M05}

\begin{abstract}
In this paper we extend the definition of a separator
of a point $P$ in $\pr^n$ to a fat point $P$ of multiplicity $m$.
The key idea in our definition is to compare the
fat point schemes
 $Z = m_1P_1 + \cdots + m_iP_i + \cdots +  m_sP_s \subseteq \pr^n$
and $Z' = m_1P_1 + \cdots + (m_i-1)P_i + \cdots + m_sP_s$. We
associate to $P_i$ a tuple of positive integers of length $\nu =
\deg Z - \deg Z'$. We call this tuple the degree of the minimal separators
of $P_i$ of multiplicity $m_i$, and we denote it by
$\deg_Z(P_i) = (d_1,\ldots,d_{\nu})$.  We show that if one knows
$\deg_Z(P_i)$ and the Hilbert function of $Z$, one will also know
the Hilbert function of $Z'$. We also show that the entries of
$\deg_Z(P_i)$ are related to the shifts in the last syzygy module
of $I_Z$. Both results generalize well known results about reduced
sets of points and their separators.
\end{abstract}
\maketitle


\section{Introduction}
Given a finite set of reduced points $X$ in $\pr^n$,
it is a classical idea to derive either
algebraic or geometric information about $X$ by using
the notion of a separator.  Our goal in this paper
is to extend the definition of a separator
so that it also includes the class of non-reduced sets points
which are usually called fat points.

Let $k$ be an algebraically closed field of characteristic zero.
A hypersurface defined by the homogeneous form $F \in R = k[x_0,\ldots,x_n]
= k[\pr^n_k]$ is said to
be a {\it separator} of $P \in X$ if $F(P) \neq 0$, but $F(Q) = 0$
for all $Q \in X\setminus\{P\}$, i.e., the hypersurface defined by $F$
passes through all the points of $X$ but $P$.  The {\it degree
of a point $P$ in $X$} is then defined to be
\[\deg_X(P) := \min\{\deg F ~|~ \mbox{$F$ is a separator of $P$}\}.\]

Separators first appeared in Orecchia's work \cite{Or1} on
the conductor of a set of points, although
the term separator does not appear until the paper
of Geramita, Kreuzer, and Robbiano \cite{GKR}.  Orecchia showed
that the conductor of the coordinate ring $A$ of a finite set
of reduced points $X = \{P_1,\ldots,P_s\}$, that is, the largest ideal
of $A$ that coincides with its extension in the integral
closure of $A$,  is generated by forms whose degrees are in the set
$\{\deg_X(P_1),\ldots,\deg_X(P_s)\}$.  For this reason, the
degree of a point $P$ in $X$ is sometimes called the {\it conductor
degree.}
 Geramita, Kreuzer, and Robbiano \cite{GKR} introduced
separators to study sets of points with the Cayley-Bacharach property.
Later investigations of the properties of separators
have included the work of Bazzotti \cite{Ba1}, Beccari and Massaza
\cite{BM2}, Sabourin
\cite{Sa}, and Sodhi \cite{S}. The definition of separators has
also been generalized to different contexts.  For example,
Bazzotti and Casanellas defined a separator for
reduced points on a surface \cite{BC}, while the authors have
studied separators of reduced sets of points in a multiprojective
space (see \cite{GV3,GV4,Ma3}).  The paper of Abbott, Bigatti,
Kreuzer, and Robbiano \cite{ABKR} contains a discussion
on how to compute the separators of a set of points.

Of particular importance to this paper are the results
of Geramita, Maroscia, and Roberts \cite{GMR} and
Abrescia, Bazzotti, and the second author \cite{ABM}.
If $X$ is a reduced set of points in $\pr^n$, and $d = \deg_X(P)$,
then Geramita, et al. showed that the Hilbert function
of $X \setminus \{P\}$ can be determined by knowing
the Hilbert function of $X$ and the value of $d$.  This
result nicely illustrates the idea that a separator gives
information about passing from $X$ to a subset of the
type $X \setminus \{P\}$.  Abrescia, et al. then
found a relationship between the
shifts in the last syzygy module of $I_X$ and the degree of
a point.  This result, originally only proved for points in $\pr^2$,
was independently extended to $\pr^n$ by the
second author \cite{Ma3} and Bazzotti \cite{Ba1}.

In the above cited work, the sets of points
being considered
are almost always a reduced set of points.  The work of Geramita,
Kreuzer, and Robbiano \cite{GKR}, Kreuzer \cite{K}, and Kreuzer
and Kreuzer \cite{KK} relaxed this condition and studied
zero dimensional subschemes $Z$, and considered subschemes
of colength 1, i.e., zero dimensional subschemes $Y \subseteq Z$
such that $\deg Y = \deg Z - 1$.  In this paper, however,
we are interested in the case that both $Y$ and $Z$ are sets
of fat points ($\deg Y = \deg Z -1$ is rarely true in this case), and
to define separators of fat points in this context.
If $X = \{P_1,\ldots,P_s\}$ is a set of reduced points in $\pr^n$,
and $m_1,\ldots,m_s$ are positive integers, then let $Z$ be the
scheme defined by $I_Z=I_{P_1}^{m_1}
\cap \cdots \cap I_{P_s}^{m_s}$ where each $I_{P_i}$ is the defining
ideal of the point $P_i$.  The scheme
$Z$, which we shall denote by
$Z=m_1P_1+\cdots+m_sP_s$, is usually called a {\it set of fat points} of $\pr^n$.

We want to define a separator of a fat point so that we recover
fat point analogs of the results of Geramita, et al. and
Abrescia, et al., that were mentioned above.  The key insight
that we need to carry out this program is to view
passing from $X$ to $X \setminus \{P\}$ as ``reducing'' the
multiplicity of $P$ by one, as opposed to ``removing'' the point $P$
from $X$.  This point-of-view appears to be the correct perspective
in order to get the desired generalizations.

Once we dispense with the preliminaries in Section 2, in Section 3
we introduce our definition of a separator for fat points.
In keeping with our idea of dropping the multiplicity of a point
by one, let $Z = m_1P_1 + \cdots + m_iP_i + \cdots m_sP_s$ and $Z'
= m_1P_1 + \cdots + (m_i-1)P_i + \cdots + m_sP_s$.  A {\it
separator} of the point $P_i$ of $Z$ of multiplicity $m_i$ is
any form $F$ such that $F \in I_{Z'} \setminus I_Z$. In Theorem
\ref{theorem1} we show that there exists a set of $\nu = \deg Z -
\deg Z'$ separators of the point $P_i$ of multiplicity $m_i$, say
$\{F_1,\ldots,F_\nu\}$, such that the ideal $I_{Z'}/I_Z$ is
minimally generated by
$(\overline{F}_1,\ldots,\overline{F}_{\nu})$ in the ring $R/I_Z$.
The {\it degree of the minimal separators of the fat point $P_i$ of
multiplicity $m_i$}, which is denoted $\deg_Z(P_i)$, is
the $\nu$-tuple $(\deg F_1,\ldots,\deg F_\nu)$.

In Section 4 and Section 5 we use our new definition
to prove fat point analogs of the results mentioned above.
In particular, we prove the following results:

\begin{theorem}
Let $Z$ and $Z'$ be the fat point schemes in $\pr^n$ defined
as above, and suppose $\deg_Z(P_i) = (d_1,\ldots,d_{\nu})$
where $\nu = \deg Z - \deg Z'$.
\begin{enumerate}
\item[(a)] \emph{[{Theorem \ref{hilbertfunction}}]}
 For all $t \in \N$,
\[\Delta H_{Z'}(t)  = \Delta H_Z(t) - |\{d_j \in (d_1,\ldots,d_{\nu}) ~|~ d_j =t \}|\]
where $\Delta H_{Y}$ denotes the first difference Hilbert function
of $Y$.
\item[(b)] \emph{[{Theorem \ref{teofat}}]}
If
\[0 \rightarrow \mathbb{F}_{n-1} \rightarrow
\cdots \rightarrow \mathbb{F}_0 \rightarrow I_Z \rightarrow 0 \]
is a minimal graded free resolution of $I_Z$, then the last syzygy
module has the form
\[\mathbb{F}_{n-1} = \mathbb{F}'_{n-1}\oplus R(-d_1-n) \oplus R(-d_2-n)
\oplus \cdots \oplus R(-d_{\nu}-n).\]
\end{enumerate}
\end{theorem}
\noindent
As an interesting corollary of Theorem 1.1 (b), we note
that if $m = \max\{m_1,\ldots,m_s\}$ is the maximum of the multiplicities
of a set of fat points in $\pr^n$, then $\rk {\mathbb{F}_{n-1}} \geq \binom{m+n-2}{n-1}$.
See Corollary \ref{rank} for more details.

We end our paper in Section 6 by calculating $\deg_Z(P)$
when $Z$ is a special class of fat points.  We show that
if $Z$ is a {\it homogeneous} set of fat points, i.e., $m_1 =
\cdots = m_s$, whose support is a complete intersection, then for
every point $P$ in the support of $Z$, the degree of the minimal
separators of the fat point $P$ of multiplicity $m$ is the same
(see Theorem \ref{degCI}).  This result can be viewed as a
Cayley-Bacharach type of result since a set of reduced points has
the {\it Cayley-Bacharach property } if and only if the degree of every point in $X$ is the
same.  The results of this section extend our understanding of fat
points in special position (see, for example,
\cite{GV1,GV2} and references there within).

\noindent
{\bf Acknowledgments.} The authors thank A.V. Geramita,
B. Harbourne, and A. Ragusa
for their comments on earlier versions of this paper.
The third author acknowledges the support of NSERC.


\section{Preliminaries and notation}

In this section we collect together some well known results which
we shall need;
we continue to use the notation and definitions from the introduction.

Let $Z = m_1P_1 + \cdots + m_sP_s$ be a set of fat points
in $\pr^n$.  The positive integers $m_1,\ldots,m_s$ are called
the {\it multiplicities}.    If $m_1 = \cdots = m_s = m$, then we
refer to $Z$ as a {\it homogeneous scheme of fat points},
otherwise $Z$ is {\it non-homogeneous}.
The set of reduced points $X=\{P_1,\dots,P_s\}$ is called the {\it
support} of $Z$, and is denoted by $\supp(Z)$.
The {\it degree} of the fat point
scheme  $Z = m_1P_1 + \cdots + m_sP_s \subseteq \pr^n$
is  given by the formula $\deg Z := \sum_{i=1}^s \binom{m_i+ n - 1}{n}.$

The defining ideal of $Z$, denoted $I_Z$, is a homogeneous
ideal in the ring $R = k[x_0,\ldots,x_n]$.
The {\it Hilbert function of $Z$}, denoted $H_Z$, is the
numerical function $H_Z:\N \rightarrow \N$ defined by
\[H_Z(t) := \dim_k (R/I_Z)_t = \dim_k R_t - \dim_k (I_Z)_t ~~\mbox{for
$t \in \N$}.\]
The first difference function of $Z$, denoted $\Delta H_Z$,
is defined by
\[\Delta H_Z(t) := H_Z(t) - H_Z(t-1) ~~\mbox{where $H_Z(t) = 0$
if $t < 0$}.\]
The eventual value
of $H_Z$ is given by the degree of $Z$:

\begin{lemma} \label{eventualhilbertfunction}
Let $Z \subseteq \pr^n$ be a fat point scheme.
Then $H_Z(t) = \deg Z ~~\mbox{for  all $t \gg 0$.}$
\end{lemma}

We also require information about the ideal of a single (fat)
point in $\pr^n$.

\begin{lemma}\label{primarylemma}
Let $I_P$ be the prime ideal associated to a point
$P \in \pr^n$.
\begin{enumerate}
\item[(a)] The ideal $I_P^m$ is $I_P$-primary.
\item[(b)] The minimal free graded resolution of $I_P$ has the form
\[0 \rightarrow R(-n) \rightarrow R^{\binom{n}{n-1}}(-n+1) \rightarrow
\cdots \rightarrow R^{\binom{n}{1}}(-1) \rightarrow I_P \rightarrow 0.\]
\end{enumerate}
\end{lemma}

\begin{proof}
(a) Since $I_P$ is a complete intersection, $I_P^m = I_P^{(m)}$,
the $m$-th symbolic power of $I_P$.  This fact follows from
a classical result of Zariski and Samuel \cite[Lemma 5, Appendix 6]{ZS}.
But $I_P^{(m)}$ is the $I_P$-primary part of $I_P^m$, so the conclusion
follows.

For (b), one appeals to the Koszul resolution.
\end{proof}


\section{Defining separators of fat points }

In this section we extend the definitions of a separator of a reduced point $P$ in $\pr^n$
and the degree of $P$ in a set of points to the case of fat points.
At the heart of our definition is the point-of-view that
the comparison of the reduced sets of points  $X$ and  $X \setminus \{P\}$
used to define separators should be
seen as ``reducing'' the multiplicity of the point $P$
by one, as opposed to ``removing'' the point $P$ from $X$.
It is this feature, i.e., the idea of reducing the multiplicity
of the fat point by one, that we will generalize when defining
a separator for a fat point.

The following convention shall be useful throughout this paper.

\begin{convention}\label{convention}
Consider the fat point scheme
\[Z := m_1P_1 + \cdots + m_iP_i + \cdots + m_sP_s \subseteq \pr^n,\]
and fix a point $P_i \in \supp(Z)$.  We then let
\[Z':= m_1P_1 + \cdots
+ (m_i-1)P_i + \cdots + m_sP_s,\]
denote the fat point scheme obtained by reducing the multiplicity of
$P_i$ by one.  If $m_i-1 = 0$, then the point $P_i$ does
not appear in the support of $Z'$.
\end{convention}

Note that when $m_j = 1$ for $j = 1,\ldots,s$, then $Z$ is
simply the reduced set of points $X = \supp(Z)$, and $Z'$
is $X \setminus \{P_i\}$, i.e., we revert to
the original context in which separators were defined.
A separator will now be defined in terms
of forms that pass through $Z'$ but not $Z$.

\begin{definition}  Let $Z = m_1P_1 + \cdots + m_iP_i+ \cdots + m_sP_s$ be a set of fat points in
$\pr^n$.  We say that $F$ is a {\it separator of the point $P_i$
of multiplicity $m_i$} if $F \in I_{P_i}^{m_i-1} \setminus
I_{P_i}^{m_i}$ and $F \in I_{P_j}^{m_j}$ for all $j \neq i$.
\end{definition}

If $F$ is a separator of the point $P_i$
of multiplicity $m_i$, then $F \in I_{Z'} \setminus I_Z$.
Thus, to compare $Z$ and $Z'$, we need to compare the ideals
$I_Z$ and $I_{Z'}$.  We can do this algebraically by investigating
the ideal $I_{Z'}/I_Z$ in the ring $R/I_Z$.

\begin{theorem}  \label{theorem1}
Let $Z$ and $Z'$ be the fat point schemes of Convention \ref{convention}.
Then there exists $\nu = \deg Z - \deg Z'$ homogeneous polynomials $\{F_1,\ldots,F_{\nu}\}$
such that
\begin{enumerate}
\item[(a)] each $F_i$ is a separator of $P_i$ of multiplicity $m_i$, and
\item[(b)] in the ring $R/I_Z$, the ideal
\[I_{Z'}/I_Z = (\overline{F}_1,\ldots,\overline{F}_{\nu})
~~\mbox{where $\overline{F}_i$ denotes the class of $F_i$.}\]  Furthermore,
these polynomials form a minimal set of generators, where by minimal
we mean that no set of cardinality less than $\nu$ generates
$I_{Z'}/I_Z$.
\end{enumerate}
\end{theorem}

\begin{proof} Because $I_{Z'}/I_Z$ is an ideal in the ring $R/I_Z$,
there exists $F_1,\ldots,F_s \in R$ such that $I_{Z'}/I_Z =
(\overline{F}_1,\ldots,\overline{F}_s)$. Moreover, because $R/I_Z$
is a Noetherian ring, we can assume that this $s$ is minimal, that
is, for any set $\{G_1,\ldots,G_t\}$ with $t < s$, then
$I_{Z'}/I_Z \neq (\overline{G}_1,\ldots,\overline{G}_t)$. Because
each $\overline{F}_j \neq 0$, this means that $F_j \not\in I_Z$.
However, $F_j \in I_{Z'}$. So, this implies that each $F_j$ is
a separator of $P_i$ of multiplicity $m_i$.  To complete the
proof, it suffices to show that $s = \deg Z - \deg Z'$.

Let $P = P_i$ and $m =m_i$.  After a linear change of variables,
we can assume that $P = [1:0:\cdots:0]$, and hence $I_P =
(x_1,\ldots,x_n)$. We can also assume that the hyperplane defined
by $L = x_0$ does not pass through any of the points of $\supp(Z)$.

We first show that $s \leq \deg Z - \deg Z'$. For all non-negative
integers $t$ we have the following short exact sequence of vector
spaces
\[0 \rightarrow (I_{Z'}/I_Z)_t \rightarrow (R/I_Z)_t \rightarrow (R/I_{Z'})_t \rightarrow 0\]
where $(M)_t$ denotes the vector space of
degree $t$ elements in $M$. Hence,
\[\dim_k (I_{Z'}/I_Z)_t = \dim_k (R/I_Z)_t - \dim_k (R/I_{Z'})_t ~~\mbox{for all $t \geq 0$.}\]
By Lemma \ref{eventualhilbertfunction},
$\dim_k (R/I_Z)_ t = \deg Z$, and $\dim_k
(R/I_{Z'})_t = \deg Z'$ for all $t \gg 0$. Hence $\dim_k (I_{Z'}/I_Z)_t = \deg Z -
\deg Z'$ for all $t \gg 0$.   Fix a $t$ such that  $\dim_k
(I_{Z'}/I_Z)_t = \deg Z - \deg Z'$ and set $t_i = t- \deg F_i$ for
each $i=1,\ldots,s$.  If necessary, we can also
take $t$ large enough so that $t_i > 0$ for all $i$.
Since $L = x_0$ is a nonzero divisor on
$R/I_Z$, each $\overline{x_0^{t_i}F}_i \neq \overline{0}$ in
$R/I_Z$.  Also note that
for each $i=1,\ldots,s$,  we have $\overline{x_0^{t_i}F}_i \in
(I_{Z'}/I_Z)_t$.

We claim that
$\left\{\overline{x_0^{t_1}F}_1,\ldots,\overline{x_0^{t_s}F}_s\right\}$ is a
linearly independent set of forms in $(I_{Z'}/I_Z)_t$, whence $s
\leq \deg Z - \deg Z'$. If necessary, relabel the $F_i$'s so that
$\deg F_1 \leq \deg F_2 \leq \cdots \leq \deg F_s$. Suppose that there
exists $c_1,\ldots,c_s$ in $k$, not all zero, such that
\[c_1\overline{x_0^{t_1}F}_1+ \cdots + c_s\overline{x_0^{t_s}F}_s = \overline{c_1x_0^{t_1}F_1 +
\cdots +c_sx_0^{t_s}F_s} = \overline{0}.\]
Let $r$ be the largest
integer in $\{1,\ldots,s\}$ such that $c_r \neq 0$.  Hence
\begin{eqnarray*}
 \overline{c_1x_0^{t_1}F_1 +
\cdots +c_sx_0^{t_s}F_s} &=&  \overline{c_1x_0^{t_1}F_1 + \cdots
+c_rx_0^{t_r}F_r}\\
& =& \overline{x_0^{t_r}}
 \overline{(c_1x_0^{t_1-t_r}F_1 +
\cdots +c_rF_r)} = \overline 0.
\end{eqnarray*}
Note that by our relabeling,
$t_i-t_r \geq 0$ for $i=1,\ldots,r$.  Because $x_0$ is a nonzero
divisor on $R/I_Z$, we must have $c_1x_0^{t_1-t_r}F_1 + \cdots
+c_rF_r = H \in I_Z$. But this implies that
\[F_r = (c_r)^{-1}c_rF_r = (c_r)^{-1}
(-c_1x_0^{t_1-t_r}F_1 - \cdots
-c_{r-1}x_0^{t_{r-1}-t_r}F_{r-1} + H).\]
But then
$\overline{F_r} \in
(\overline{F}_1,\ldots,\overline{F}_{r-1},\overline{F}_{r+1},
\ldots,\overline{F}_s)$, whence
\[(\overline{F}_1,\ldots,\overline{F}_{r-1},\overline{F}_{r+1},
\ldots,\overline{F}_s) = (\overline{F}_1,\ldots,\overline{F}_{r},\ldots,\overline{F}_s),\]
thus contradicting the minimality of $s$.

We now show that if $s < \deg Z - \deg Z'$, we can derive a
contradiction, and hence $s= \deg Z - \deg Z'$.  As above, fix $t$
to be any integer such that $\dim_k (I_{Z'}/I_Z)_t = \deg Z - \deg
Z'$.  If $s < \deg Z - \deg Z'$, then there exists some
$\overline{H} \in (I_{Z'}/I_Z)_t$ that is not in the span of
$\left\{\overline{x_0^{t_1}F}_1,\ldots,\overline{x_0^{t_s}F}_s\right\}$.
On the other hand, because $\overline{H} \in (I_{Z'}/I_Z)_t$,
there exist homogeneous forms $G_1,\ldots,G_s$  in $R$ such that
\[\overline{H} = \overline{G_1F_1 + \cdots +G_sF_s} ~~\mbox{with $\deg G_i = t -\deg F_i$}.\]
Each $G_i$ can be rewritten as
\[G_i = c_ix_0^{t-\deg F_i} + G'_i(x_0,\ldots,x_n)\]
where $G'_i = G'_i(x_0,\ldots,x_n) \in (x_1,\ldots,x_n) = I_P$.
We then have
\[\overline{G_iF_i} = \overline{c_ix_0^{t-\deg F_i}F_i} + \overline{G'_iF_i}
= \overline{c_ix_0^{t-\deg F_i}F_i}\] since $G'_iF_i \in I_Z$ for
all $i$.  To see this, note that for any $P_j \in \supp(Z)
\setminus \{P\}$, we already have $F_i \in I_{P_j}^{m_j}$, and
thus $G'_iF_i \in   I_{P_j}^{m_j}$.  On the other hand, since
$G'_i \in I_P$ and $F_i \in I_P^{m-1}$, we get $G'_iF_i \in
I_P^m$.  As a consequence
\[\overline{H} = \overline{G_1F_1 + \cdots +G_sF_s} = \overline{c_1x_0^{t-\deg F_1}F_1
+ \cdots + c_sx_0^{t-\deg F_s}F_s}.\] But this implies that
$\overline{H}$ is in the span of
$\{\overline{x_0^{t_1}F}_1,\ldots,\overline{x_0^{t_s}F}_s\}$,
contradicting our choice of $\overline{H}$.  Hence $s = \deg Z -
\deg Z'$, as desired.
\end{proof}

\begin{remark}
The number $\nu = \deg Z - \deg Z'$ can be computed directly
from the degree formula;  precisely
\begin{eqnarray*}
\deg Z - \deg Z' &= &\deg  m_iP_i - \deg (m_i-1)P_i \\
&=& \binom{m_i+n-1}{n}-\binom{m_i-1+n-1}{n} =
\binom{m_i+n-2}{n-1}.
\end{eqnarray*}
\end{remark}

In light of the above theorem, we can introduce a minimal
set of separators:

\begin{definition} Let $Z$ and $Z'$ be as in Convention \ref{convention}.
 If $\{F_1,\ldots,F_{\nu}\}$ is a set of polynomials that
satisfies conditions (a) and (b) of Theorem \ref{theorem1}, then we call
$\{F_1,\ldots,F_{\nu}\}$ a {\it minimal set of separators of $P_i$
of multiplicity $m_i$}.
\end{definition}

Our next step is to use this minimal set of separators to
develop a fat point analog for the degree of a point.
We
begin with a lemma.

\begin{lemma}  \label{lemma1}
Let $Z$ and $Z'$ be as in Convention \ref{convention}
with associated
ideals $I_Z$ and $I_Z'$, respectively.
Suppose that $\{F_1,\ldots,F_{\nu}\}$ is a minimal set of
separators of $P_i$ of multiplicity $m_i$.  Then, for all $t \geq 0$,
\[ \dim_k (I_{Z'}/I_Z)_t = |\{F_i ~|~ \deg F_i \leq t\}|.\]
\end{lemma}

\begin{proof}  Assume that $P := P_i = [1:0:\cdots:0]$ and that the hyperplane
defined by $L = x_0$ is a nonzero divisor on $R/I_Z$.  We can now
argue as in the proof of Theorem \ref{theorem1} to get the
conclusion. Indeed, fix any integer $t$, and let $F_1,\ldots,F_r$
be all the forms in the set $\{F_1,\ldots,F_\nu\}$ with $\deg F_i \leq t$.
Then $\left\{\overline{x_0^{t-\deg
F_1}F}_1,\ldots,\overline{x_0^{t-\deg F_r}F}_r\right\}$ is a linearly
independent set of elements in $(I_{Z'}/I_Z)_t$.

Furthermore, this set must span $(I_{Z'}/I_Z)_t$.  Indeed, for any
$\overline{H} \in (I_{Z'}/I_Z)_t$, there exists homogeneous
forms $G_1,\ldots,G_r$ such that
\[\overline{H} = \overline{G_1F_1 + \cdots + G_rF_r} ~~\mbox{with $\deg G_i = t-\deg F_i$}.\]
Note, by degree considerations, we do not need to concern
ourselves with the forms $F_{r+1},\ldots,F_\nu$.  Just as in proof
of Theorem \ref{theorem1}, we rewrite each $G_i$ as $G_i =
c_ix_0^{t-\deg F_i}  + G'_i$ with $G'_i \in I_P$. This then implies
that
\[\overline{H} = \overline{c_1x_0^{t-\deg F_1}F_1 + \cdots + c_rx_0^{t-\deg F_r}F_r},\]
i.e., $\overline{H}$ is in the span of $\left\{\overline{x_0^{t-\deg
F_1}F}_1,\ldots,\overline{x_0^{t-\deg F_r}F}_r\right\}$.

Because  $\left\{\overline{x_0^{t-\deg
F_1}F}_1,\ldots,\overline{x_0^{t-\deg F_r}F}_r\right\}$ is a basis for
$(I_{Z'}/I_Z)_t$, the conclusion now follows.
\end{proof}

\begin{theorem}  \label{theorem2}
Let $Z$ and $Z'$ be as in Convention \ref{convention}.
Suppose that $\{F_1,\ldots,F_{\nu}\}$ and $\{G_1,\ldots,G_{\nu}\}$
are two minimal sets of separators of $P_i$ of multiplicity $m_i$.
Relabel the $F_i$'s so that $\deg F_1 \leq \cdots \leq \deg
F_{\nu}$, and similarly for the $G_i$'s.  Then
\[(\deg F_1,\ldots,\deg F_{\nu}) = (\deg G_1,\ldots,\deg G_{\nu}).\]
\end{theorem}

\begin{proof}
This follows immediately from Lemma \ref{lemma1} since we must
have
\[|\{F_i ~|~ \deg F_i \leq t \}| = |\{G_i ~|~ \deg G_i \leq t\}|\]
for all integers $t \geq 0$.
\end{proof}

Using Theorem \ref{theorem2}, we can
define a fat point analog for the degree of a point.

\begin{definition}  Let $\{F_1,\ldots,F_{\nu}\}$ be any minimal
set of separators of $P_i$ of multiplicity $m_i$, and relabel so
that $\deg F_1 \leq \cdots \leq \deg F_{\nu}$.  Then the {\it
degree of the minimal separators of $P_i$ of multiplicity $m_i$},
denoted $\deg_Z(P_i)$, is the tuple
\[\deg_Z(P_i) = (\deg F_1,\ldots, \deg
F_{\nu}).\]
\end{definition}

\begin{remark}
When all the multiplicities of $Z$ are one, then $\nu = 1$,
and $\deg_Z(P_i) = (\deg F_1)$ where $F_1$ is a minimal
separator of $P_i$ of multiplicity of $m_i = 1$.  From the definition,
we observe that $F_1$ passes through all the points of $Z = \supp(Z)$
except the point $P_i$, i.e., $F_1$ is a minimal separator of $P_i$ in
the traditional sense.
\end{remark}

We now illustrate some of the above ideas with the following two examples.

\begin{example}\label{example1}
Suppose that $Z = mP$ is a single fat point of multiplicity $m \geq 2$ in
$\pr^n$.  We can therefore assume that $I_P = (x_1,\ldots,x_n)$,
and $I_Z = I_P^m$.  In this case, all the monomials
of degree $m-1$ in the variables $\{x_1,\ldots,x_n\}$ form
a minimal set of separators of $P$ of multiplicity $m$
since
\[I_{Z'}/I_Z = \left.\left(\left\{~~\overline{M} ~\right|~ M = x_1^{a_1}\cdots x_n^{a_n} ~~
\mbox{with $a_1 + \cdots + a_n = m-1$}\right\}\right).\]
Thus, $\deg_Z(P) = (\underbrace{m-1,\ldots,m-1}_{\binom{m+n-2}{n-1}})$.
\end{example}

\begin{example}\label{example2}
Let $F,G \in R=k[x,y,z]$ be two generic forms with $\deg F = 2$
and $\deg G = 3$.  Then $I = (F,G)$ defines a complete intersection
of six reduced points $\{P_1,\ldots,P_6\}$ in $\pr^2$ of type $(2,3)$.  Because
$I$ is a complete intersection, the ideal $I^2 = (F,G)^2$
is the defining ideal of the set of double points:
\[Z = 2P_1 + \cdots + 2P_6 \subseteq \pr^2.\]
Let $Z' = 1P_1 + 2P_2 + \cdots + 2P_6$, and let $I_Z$ and $I_{Z'}$ be the
associated ideals.  The Hilbert functions of $Z$ and $Z'$
are, respectively,
\[
\begin{array}{c|rrrrrrrr}
t & 0 & 1& 2 &3 &4 &5 & 6 & 7  \\
H_Z(t) &1 & 3 & 6 & 10 & 14 & 17 & 18 & \rightarrow\\
H_{Z'}(t)&1 & 3 & 6 & 10 & 14 & 16 & 16 & \rightarrow\\
\end{array}
\]
From the above Hilbert functions, we can determine
$\dim_k(I_{Z'}/I_Z)_t = H_Z(t) - H_{Z'}(t)$ for all $t$.  By appealing
to Lemma \ref{lemma1}, we then obtain $\deg_Z(P_1) = (5,6)$.
The connection between the Hilbert functions of $H_Z$ and $H_{Z'}$
and the tuple $\deg_Z(P_1)$ will be highlighted in the next section.
\end{example}

As we shall see in the later sections, information about
$Z'$ can be obtained from $Z$ and $\deg_Z(P_i)$.  By reiterating
this process, we can then start from any fat point scheme,
and successively reduce the multiplicity of any fat point by one to
obtain information about the subschemes of $Z$ that are also
fat point schemes.  It therefore makes sense to develop
some suitable notation and definitions to carry out this iteration.
We end this section by working out these details.

We begin by introducing some more notation that describes
the scheme after we have dropped the multiplicity of $P_i$
by any integer $h \in \{0,\ldots,m_i\}$.

\begin{definition}\label{iteration}
Let $Z=m_1P_{1}+ \cdots + m_sP_{s}$ be a fat point scheme in
$\pr^n$ whose support is  $X=\{P_{1},\ldots,P_{s}\}$.
If we fix an $i \in \{1,\ldots,s\}$,
then for every $h\in \{0,\ldots, m_i\}$ we define
\[{Z}_{m_{i}-h}(P_i)=m_1P_{1}+ \cdots + m_{i-1}P_{i-1}+(m_{i}-h)P_i+
m_{i+1}P_{i+1}+\cdots+m_sP_s.\]
\noindent
We shall write
${Z}_{m_{i}-h}$ when $P_i$ is understood.
\end{definition}

Note that what we called $Z$ and $Z'$ in Convention \ref{convention}
are denoted $Z_{m_i}$ and $Z_{m_i-1}$ with respect to the new notation.
If $h=m_i$, then
\[Z_0={Z}_{0}(P_i) =m_1P_1+\cdots+m_{i-1}P_{i-1}+m_{i+1}P_{i+1}+\cdots+m_{s}P_s\]
is a scheme of fat points whose support is $\supp(Z) \setminus
\{P_{i}\}.$ We can now introduce the degree of the minimal separators
at various levels, where the level keeps track of how much we have
reduced the multiplicity.

\begin{definition}
Suppose that $Z = m_1P_1 + \cdots + m_iP_i + \cdots + m_sP_s$. For
$h=1,\dots,m_i$, the {\it degree of the minimal separators of
$P_i$ of multiplicity $m_i$ at level $h$}, is
$\deg_{Z_{m_i-h+1}}(P_i)$.
\end{definition}

When $h=1$, $\deg_{Z_{m_i-h+1}}(P_i) = \deg_Z(P_i)$, so we can
view $\deg_Z(P_i)$ as the degree of the minimal separators of
$P_i$ of multiplicity $m_i$ at level $1$. We can now combine all
degrees at each level to define the minimal separating set of a
fat point.

\begin{definition}  Let $Z = m_1P_1 + \cdots + m_iP_i + \cdots + m_sP_s$.
The {\it minimal separating set of the fat point $m_iP_i$} is the
set
\[{\sf DEG}_Z(m_iP_i) = \{\deg_{Z_1}(P_i),\ldots,\deg_{Z_{m_i}}(P_i)\}.\]
\end{definition}

\begin{remark}
Note that $\deg_{Z_{1}}(P_i)$ has only one entry and it represents
the minimal degree of a form that passes through all the points
$P_j$ of $Z$ with multiplicity at least $m_j$ with $j\ne i$, but
not through $P_i$. When $m_i = 1$, the minimal separating set of
the fat point $1P_i$, which in this case is a reduced point, is
the set ${\sf DEG}_Z(1P_i) = \{\deg_{Z_1}(P_i)\}$, and this
corresponds to the separator degree of a reduced point $P_i$  as
given in the introduction.
\end{remark}

\section{Hilbert functions and separators}

In this short section we explain how to use $\deg_Z(P_i)$
to compare the Hilbert functions of $Z$ and $Z'$.
We continue to use Convention \ref{convention}.
Our main result specializes
to a result of Geramita, et al. \cite{GMR} when
all the multiplicities are one.

At the core of the following theorem is Lemma \ref{lemma1}
which computes the dimension of $(I_{Z'}/I_Z)_t$ for all $t$.
Recall that $\Delta H_Z$ denotes the first difference function.
In what follows, we write
$a \in (a_1,\ldots,a_n)$ to mean that $a$ appears in the tuple
$(a_1,\ldots,a_n)$.

\begin{theorem}  \label{hilbertfunction}
Let $Z$ and $Z'$ be as in Convention
\ref{convention}.  Suppose that $\deg_Z(P_i) = (d_1,\ldots,d_{\nu})$
where $\nu = \deg Z - \deg Z'$.
Then for all $t \in \N$,
\[\Delta H_{Z'}(t)  = \Delta H_Z(t) - |\{d_j \in (d_1,\ldots,d_{\nu}) ~|~ d_j =t \}|.\]
\end{theorem}

\begin{proof}
For each $t \in \N$, the Hilbert functions of $Z$ and $Z'$ in degree
$t$  are related
via the following short exact sequence of vector spaces:
\[0 \rightarrow (I_{Z'}/I_Z)_t \rightarrow
(R/I_Z)_t \rightarrow (R/I_{Z'})_t \rightarrow 0. \]
Thus,
\begin{eqnarray*}
\Delta H_{Z'}(t) & = & H_{Z'}(t) - H_{Z'}(t-1) \\
& = & (H_{Z}(t) - \dim_k (I_{Z'}/I_Z)_t) - (H_Z(t-1) - \dim_k (I_{Z'}/I_Z)_{t-1}) \\
& = & \Delta H_{Z}(t) -  (\dim_k (I_{Z'}/I_Z)_t -  \dim_k (I_{Z'}/I_Z)_{t-1}).
\end{eqnarray*}
The conclusion now follows from Lemma \ref{lemma1} since
\begin{eqnarray*}
\dim_k (I_{Z'}/I_Z)_t -  \dim_k (I_{Z'}/I_Z)_{t-1} &=& |\{d_j \in \deg_Z(P_i) ~|~
d_j \leq t \}| - |\{d_j \in \deg_Z(P_i) ~|~
d_j \leq t-1 \}| \\
& = &  |\{d_j \in (d_1,\ldots,d_{\nu}) ~|~ d_j =t \}|,
\end{eqnarray*}
thus completing the proof.
\end{proof}

\begin{remark}
Suppose that one knows
two of the following three pieces of information: (1) $H_Z$, (2) $H_{Z'}$, and
$(3)$ $\deg_Z(P_i)$.  It follows from Theorem \ref{hilbertfunction}
that one can also determine the third piece of information.
\end{remark}

\begin{example}  In Example \ref{example1} we calculated $\deg_Z(P)$
when $Z = mP \subseteq \pr^n$.  We use this
information to find the Hilbert function of $Z = 3P$ in
$\pr^2$.  By Theorem \ref{hilbertfunction}
\[\Delta H_{2P}(t) =
\left\{
\begin{array}{ll}
\Delta H_{3P}(t) & \mbox{if $t \neq 2$} \\
\Delta H_{3P}(t) - 3 & \mbox{if $t = 2$}
\end{array}
\right.\]
because $\deg_{3P}(P) = (2,2,2)$.  We now need to find the
Hilbert function of $\Delta H_{2P}$.  Again, appealing to
Theorem \ref{hilbertfunction}, we get
\[\Delta H_{P}(t) =
\left\{
\begin{array}{ll}
\Delta H_{2P}(t) & \mbox{if $t \neq 1$} \\
\Delta H_{2P}(t) - 2 & \mbox{if $t = 1$}
\end{array}
\right.\]
because $\deg_{2P}(P) = (1,1)$.
Since $H_P(t) = 1$ for all $t \in \N$, we use the above
expressions to find
\[\Delta H_{3P} : 1 ~ 2 ~ 3 ~ 0 ~ \rightarrow.\]
It follows that
this recursive procedure can be used to find the Hilbert
function of any single fat point in any projective space.
Indeed, when $Z = mP \subseteq \pr^2$, this procedure
recovers the well known result that
\[\Delta H_{mP} : 1 ~ 2 ~ 3 ~ \cdots ~ m-1 ~ m ~ 0 ~ \rightarrow.\]
\end{example}

When we specialize to the case of reduced points
we recover a result of Geramita, Maroscia, and Roberts.

\begin{corollary}[{\cite[Lemma 2.3]{GMR}}]\label{hilbertonept}
Let $X \subseteq \pr^n$ be a reduced set
of points, and suppose that $P \in X$.  If $X' = X \setminus \{P\}$, then
\[ \Delta H_{{X'}}(t) = \begin{cases}
              \triangle H_{{X}}(t)   & t \neq \deg_X(P) \\
               \triangle H_{{X}}(t)-1 & t = \deg_X(P).
              \end{cases}
\]
\end{corollary}

In the same paper, Geramita, et al. defined a {\it permissible
value} (see \cite[Definition 4.1]{GMR}) and showed that
the degree of every point $P$ is $X$ is a permissible value.  We round
out this section by generalizing the notion of a permissible value
and show that the degree of a minimal set of separators of $P$ of multiplicity
$m$ is also an example of this generalized permissible value.

\begin{definition}\label{sh}
Let $H = \{ b_t \}$, $t \geq 0$ be a zero-dimensional
differentiable $O$-sequence.   That is, $H$ is the Hilbert
function of a zero-dimensional scheme, and its first difference is
also an $O$-sequence (see \cite{GMR}, for example, for the
definition of an $O$-sequence).  Equivalently, if $b_1 = n+1$,
then $H$ is a zero-dimensional differentiable $O$-sequence if its
first difference function $\Delta H$ is the Hilbert function of an
artinian quotient of $k[x_1,\ldots,x_n]$.  Let $\underline{d} =
(d_1,\ldots,d_{\tau})$ be any $\tau$-tuple of positive
integers with $\tau \geq 1$ and
$d_1 \leq \cdots \leq d_{\tau}$. We say that $\underline{d}$ is a
\emph{permissible vector} of length $\tau$ for  $H$ if
$$H_{\underline{d}} =  \left\{b_t - |\{d_j \in (d_1,\ldots,d_{\tau}) ~|~ d_j \leq t\}|\right\}$$
is again a zero-dimensional differentiable $O$-sequence.
The set of all permissible vectors of length $\tau$ with
respect to $H$ shall be denoted by $S_{H,\tau}$.
\end{definition}

Theorem \ref{hilbertfunction} implies that $\deg_Z(P_i)$ is a permissible
vector of $H_Z$.

\begin{corollary}\label{fatpermissH}
Let $Z$ and $Z'$ be as in Convention
\ref{convention}.  Suppose that $\deg_Z(P_i) = (d_1,\ldots,d_{\nu})$
where $\nu = \deg Z - \deg Z'$.
Then
\[\deg_Z(P_i) \in S_{H_Z,\nu}.\]
\end{corollary}

\begin{proof} We use the formula for $\Delta H_{Z'}$ in Theorem \ref{hilbertfunction}
to calculate $H_{Z'}$:
\[H_{Z'}(t) = H_Z(t) - |\{d_j \in \deg_Z(P_i) ~|~ d_j \leq t\}|.\]
Since $H_Z$ and $H_{Z'}$ are zero-dimensional differentiable $O$-sequences,
it follows that $\deg_Z(P_i)$ is a permissible vector of length $\nu$ of $H_Z$.
\end{proof}


\section{The degree of a separator and the minimal resolution}

As evident in the previous section, if one knows some information
about $Z$ and the tuple $\deg_Z(P_i)$, one can also obtain
information about $Z'$.  It is therefore useful to know
how to find $\deg_Z(P_i)$.  Abrescia, Bazzotti, and the second author \cite{ABM}
showed that in the case of reduced points in $\pr^2$
(and extended to $\pr^n$ in \cite{Ma3} and \cite{Ba1}), the degree of a point
in $X$ is related to a shift in the last syzygy module in the resolution
of $I_X$.
In this section we will prove a similar result
about $\deg_Z(P_i)$: the entries in this tuple are related
to the shifts in the last syzygy module of the resolution
of $I_Z$.

Before arriving at our main result, we will require a technical lemma
that will be used in the induction step of our next theorem.

\begin{lemma} \label{idealepunto}
Let $Z$ and $Z'$ be as in Convention \ref{convention}.
Let $\{F_1,\ldots,F_{\nu}\}$ be a minimal set of separators
of $P_i$ of multiplicity $m_i$, and furthermore, suppose that
the separators have been relabeled so
that $\deg F_1 \leq \cdots \leq \deg F_{\nu}$.
Then
\begin{enumerate}
\item[(a)] For $j=1,\ldots,\nu$, $(I_Z,F_1,\ldots,F_{j-1}):(F_j) = I_{P_i}$.
\item[(b)] For $j=1,\ldots,\nu$, $(I_Z,F_1,\ldots,F_j)$ is
a saturated ideal.
\end{enumerate}
\end{lemma}

\begin{proof}
We set $d_j := \deg F_j$ for $j = 1, \ldots, \nu$.

(a) To prove the inclusion $I_{P_i}\subseteq (I_{Z}, F_1,\ldots,F_{j-1}):(F_j)$,
note that $F_j \in I_{P_l}^{m_l}$ for all $l \neq i$, and for
$l = i$, $F_jI_{P_i} \subseteq I_{P_i}^{m_i}$ since $F_j \in I_{P_i}^{m_i-1}$.  Hence
$F_jI_{P_i} \subseteq I_Z \subseteq (I_Z,F_1,\ldots,F_{j-1})$.

To prove the other inclusion, we do
a change of coordinates so that $P_i = [1:0:\cdots:0]$, and
so that $x_0$ is a nonzero divisor on $R/I_Z$.  Note
that $I_{P_i} = (x_1,\ldots,x_n)$.
Suppose that
$G \in (I_{Z},F_1,\ldots,F_{j-1}):(F_j)$.
So, $GF_j \in (I_Z,F_1,\ldots,F_{j-1})$.  Then there are forms
$A_1,\ldots,A_{j-1} \in R$ and $A \in I_Z$
such that
\begin{equation}\label{GF}
GF_j = A + A_1F_1 + \cdots + A_{j-1}F_{j-1} \Leftrightarrow
GF_j - (A_1F_1 + \cdots + A_{j-1}F_{j-1}) = A \in I_Z.
\end{equation}
We can take $G,A_1,\ldots,A_{j-1}$ to be homogeneous.  Furthermore,
if $\deg A  = d$, then $\deg G = d- d_j$ and
$\deg A_l = d- d_l$ for $l=1,\ldots,j-1$.
Furthermore, $d-d_l \geq 0$ for $l = 1,\ldots,j-1$
by our ordering of
the minimal separators.
We can
also write
\[G = cx_0^{d-d_j} + G'~~\mbox{and}~~A_l = a_lx_0^{d-d_l} + A'_l\]
where $G',A_1',\ldots,A_{j-1}' \in I_{P_i} = (x_1,\ldots,x_n)$.
Our goal is to show that $c=0$, whence $G \in I_{P_i}$.

It follows that $G'F_{j-1} \in I_{P_i}^{m_i}$, and similarly
$A'_lF_l \in I_{P_i}^{m_i}$ for $l=1,\ldots,j-1$.  Because
$F_1,\ldots,F_j \in I_{P_l}^{m_l}$ for $l \neq i$, we get
\[G'F_j - (A'_1F_1 + \cdots + A'_{j-1}F_{j-1}) \in I_Z.\]
If we subtract this expression from (\ref{GF}), we get
\[cx_0^{d-d_j}F_j - (a_1x_0^{d-d_1}F_1 + \cdots + a_{j-1}x_0^{d-d_{j-1}}F_{j-1}) \in I_Z.\]
But then in $(I_{Z'}/I_Z)_d$ we have
\begin{equation}\label{c}
\overline{cx_0^{d-d_j}F_j - (a_1x_0^{d-d_1}F_1 + \cdots + a_{j-1}x_0^{d-d_{j-1}}F_{j-1})}
= \overline{0}.
\end{equation}
But by adapting the proof given in Theorem \ref{theorem1},
(this is where you require that $x_0$ to be a nonzero divisor)
the elements $\left\{\overline{x_0^{d-d_1}F_1},
\ldots, \overline{x_0^{d-d_{j}}F_{j}}\right\}$ are linearly independent
in $(I_{Z'}/I_Z)_d$.  Thus equation (\ref{c}) holds only if
$c = 0$.  But this means that
$G = G' \in I_{P_i}$, as desired.

To prove (b), we do a proof by contradiction.  So,
suppose that there exists a $j$ such that
$(I_Z,F_1,\ldots,F_j)$ is not saturated.
As above, we take $P_i = [1:0:\cdots:0]$
and $x_0$ to be a nonzero divisor.
The saturation of $(I_Z,F_1,\ldots,F_j)$, denoted
$(I_Z,F_1,\ldots,F_j)^{\operatorname{sat}}$, is given by
\[(I_Z,F_1,\ldots,F_j)^{\operatorname{sat}} =
(I_Z,F_1,\ldots,F_j):(x_0,\ldots,x_n)^{\infty}.\]

Now suppose that there exists a $G \in (I_Z,F_1,\ldots,F_j)^{\operatorname{sat}}
\setminus (I_Z,F_1,\ldots,F_j)$.  It then follows that
$Gx_0^{t} \in (I_Z,F_1,\ldots,F_j)$ for $t \gg 0$.
For any $P_l \in \supp(Z) \setminus \{P_i\}$, we have
$Gx_0^{t} \in I_{P_l}^{m_l}$ since $(I_Z,F_1,\ldots,F_j) \subseteq I_{P_l}^{m_l}$.
Because $x_0$ is a nonzero divisor on $R/I_Z$, $x_0 \not\in I_{P_l}$,
Thus, no power of $x_0$ belongs to any $I_{P_l}^{m_l}$.  This
means no power of $x_0^t$ belongs to $I_{P_l}^{m_l}$, and
thus, by Lemma \ref{primarylemma}, $G \in I_{P_l}^{m_l}$ since
$I_{P_l}^{m_l}$ is a primary ideal.

On the other hand, since $(I_Z,F_1,\ldots,F_j) \subseteq I_{P_i}^{m_i-1}$,
we have $Gx_0^{t} \in I_{P_i}^{m_i-1}$, and arguing as above,
we must have $G \in I_{P_i}^{m_i-1}$.  Thus, $G \in I_{Z'}$,
or in other words, $\overline{G} \neq \overline{0}$ in $(I_{Z'}/I_Z)$.
(If $\overline{G} = {\overline{0}}$, that would mean
$G \in I_Z \subseteq (I_Z,F_1,\ldots,F_j)$, contradicting our
choice of $G$.)

We then have
\[\overline{G} = \overline{c_1x_0^{d-d_1}F_1 + \cdots + c_{\nu}x_0^{d-d_{\nu}}F_{\nu}}\]
for some constants $c_1,\ldots,c_{\nu}$, where the constant
is zero if $d-d_{\nu} < 0$.  There then must exist some $A \in I_Z$,
such that
\[G - (c_1x_0^{d-d_1}F_1 + \cdots + c_{\nu}x_0^{d-d_{\nu}}F_{\nu}) = A \in I_Z.\]
Rearranging gives us
\begin{equation}\label{G}
G = A + (c_1x_0^{d-d_1}F_1 + \cdots + c_{\nu}x_0^{d-d_{\nu}}F_{\nu})
\end{equation}
and thus,
\begin{equation}\label{firsteq}
Gx_0^t = Ax_0^t + (c_1x_0^{d-d_1+t}F_1 + \cdots + c_{\nu}x_0^{d-d_{\nu}+t}F_{\nu}).
\end{equation}
But $Gx_0^t \in (I_Z,F_1,\ldots,F_j)$, so we can also write it
as
\[Gx_0^t = B + B_1F_1 + \cdots + B_jF_j\]
with $B \in I_Z$ and $B_1,\ldots,B_j \in R$.

We can rewrite
each $B_l$ for $l=1,\ldots,j$ as
\[B_l = b_lx_0^{d-d_l+t}+B'_l ~~\mbox{with $B'_l \in I_{P_i} = (x_1,\ldots,x_n)$.}\]
Since
$B'_lF_l \in I_{P_i}^{m_i}$ and $F_l \in I_{P_r}^{m_r}$ for all $P_r \in
\supp(Z)\setminus \{P_i\}$, we can write
$Gx_0^t$ has
\begin{equation}\label{secondeq}
Gx_0^t = B' + b_1x_0^{d-d_i+t}F_1 + \cdots + b_jx_0^{d-d_j+t}F_j
~~\mbox{with $B' \in I_Z$,}
\end{equation}
that is, the terms $B'_lF_l$ get absorbed into the $B'$.
Setting the expressions (\ref{firsteq}) and (\ref{secondeq})
equal to each other and rearranging, we get
\small
\[(c_1-b_1)x_0^{d-d_1+t}F_1 + \cdots + (c_j-b_j)x_0^{d-d_j+t}F_j
+ c_{j+1}x_0^{d-d_{j+1}+t}F_{j+1} + \cdots + c_{\nu}x_0^{d-d_{\nu}+t}F_{\nu} \in I_Z.\]
\normalsize
But if we now consider the class of this element in $I_{Z'}/I_Z$,
this element is $\overline{0}$.  However
the elements $\left\{\overline{x_0^{d-d_1+t}F_1},
\ldots, \overline{x_0^{d-d_{\nu}+t}F_{\nu}}\right\}$ form a linear
independent set (we omit any term with $d-d_i+t < 0$).
So $c_1-b_1 = \cdots = c_j-b_j = c_{j+1} = \cdots = c_{\nu} = 0$,
or in other words, $c_l = b_l$ for $l =1,\ldots,j$,
and zero for the remaining $c_l$'s.  But by (\ref{G}),
this implies that $G \in (I_Z,F_1,\ldots,F_j)$
contradicting our choice of $G$.
\end{proof}

\begin{remark} A different proof of Lemma \ref{idealepunto} (b),
can be obtained by using Proposition 3.13 and Remark 3.14 in
\cite{Gu2}.
\end{remark}

\begin{remark}\label{HFquasi}
Although $(I_Z,F_1,\ldots,F_j)$ is saturated for all $j$,
it does not define a fat point scheme.  It, however,
defines a scheme of degree $\deg Z - j$.  If we let
$W_j$ define the scheme defined by this ideal,
then $W_0,\ldots,W_{\binom{m+n-2}{n-1}}$ are all the
``intermediate'' schemes between $Z'$ and $Z$,
i.e.,
\[Z'= W_{\binom{m+n-2}{n-1}} \subset \cdots \subset W_1 \subset W_0 = Z.\]
\end{remark}

We will now prove the main theorem of this section:  given a
minimal graded free resolution of $I_Z$, the
entries of $\deg_{Z}(P_i)$ appear among the degrees of the last syzygies
after shifting by $n$.

\begin{theorem}\label{teofat}
Let $Z,Z'$ be the fat point schemes of $\pr^n$ as in Convention
\ref{convention}, and suppose that $\deg_Z(P) =
(d_1,\ldots,d_{\nu})$ where $\nu = \deg Z - \deg Z'$. If
\begin{equation}\label{minres}
0 \rightarrow \mathbb{F}_{n-1} \rightarrow \cdots
\rightarrow \mathbb{F}_0 \rightarrow I_Z \rightarrow 0
\end{equation} is the minimal graded free resolution of $I_Z$, then
the last syzygy module has the form
\[\mathbb{F}_{n-1} = \mathbb{F}'_{n-1}\oplus R(-d_1-n) \oplus R(-d_2-n)
\oplus \cdots \oplus R(-d_{\nu}-n).\]
\end{theorem}

\begin{proof}
Let $F_1,\ldots,F_{\nu}$ be a minimal set of separators of $P_i$
of multiplicity $m_i$ and let
$d_r = \deg F_r$ for $r = 1,\ldots,\nu$.
Let $\mathcal{H}_0$ denote the minimal graded free resolution of $I_Z$.
We will proceed
by induction on $r$.

When $r=1$, we have the short exact sequence
\begin{equation}\label{ses1}
0\rightarrow R/((I_{Z}):(F_{1}))(-d_{1})
\xrightarrow{\times F_{1}} R/I_{Z} \rightarrow R/(I_{Z},F_{1})\rightarrow
0.
\end{equation}
By Lemma \ref{idealepunto} we have $R/((I_Z):(F_1)) = R/I_{P_i}$.  Thus,
by Lemma \ref{primarylemma}, the minimal graded free
resolution of $R/((I_{Z}):(F_{1}))(-d_{1})$ has the form
\[ \mathcal{K}_1: 0 \rightarrow R(-d_1-n) \rightarrow R^{\binom{n}{n-1}}(-d_1-n+1) \rightarrow \cdots \rightarrow R(-d_1)
\rightarrow  R/I_{P_i}(-d_{1}) \rightarrow 0.\]

If we now apply the mapping cone construction to ($\ref{ses1}$),
using the resolutions $\mathcal{K}_1$ and $\mathcal{H}_0$,
we construct a graded resolution of $(I_Z,F_1)$:
\[ \mathcal{H}:
0 \rightarrow R(-d_1-n) \rightarrow \mathbb{F}_{n-1}\oplus R^{\binom{n}{n-1}}(-d_1-n+1)
\rightarrow \cdots \rightarrow R \rightarrow R/(I_Z,F_1) \rightarrow 0.\]
The mapping cone construction gives a resolution that,  in
general, is not minimal. Since the ideal $(I_Z,F_1)$ is saturated
by Lemma \ref{idealepunto},
its projective dimension is at most $n-1$, and thus $\mathcal{H}$
is a non-minimal resolution.   So $\mathcal{H} = \mathcal{F}
\oplus \mathcal{G}$ where $\mathcal{F}$ is the minimal resolution
of $R/(I_{Z},F_1)$ and $\mathcal{G}$ is isomorphic to the trivial
complex\footnote[1]{A {\it trivial complex} is the direct sum of complexes
of the form $0 \rightarrow R \stackrel{1}{\rightarrow} R \rightarrow 0
\rightarrow 0 \rightarrow \cdots$} (see \cite[Theorem 20.2]{E}). In particular $R(-d_1-n)$
must be part of the trivial complex $\mathcal{G}$, and thus, to
obtain a minimal resolution, the term $R(-d_1-n)$ must cancel with
something in
\[ \mathbb{F}_{n-1}\oplus R^{\binom{n}{n-1}}(-d_1-n+1).\]

By degree considerations, we cannot cancel with
any of the terms of  $R^{\binom{n}{n-1}}(-d_1-n+1)$.
Thus, $\mathbb{F}_{n-1} = \mathbb{F}'_{n-1} \oplus R(-d_1-n)$,
i.e., the term $R(-d_{1}- n)$ must cancel with something in $\mathbb{F}_{n-1}$.
Note that after we cancel $R(-d_1-n)$, we get a resolution
of $(I_Z,F_1)$, that may or may not be minimal.  We let
\[ \mathcal{H}_1:
0 \rightarrow \mathbb{F}'_{n-1}\oplus R^{\binom{n}{n-1}}(-d_1-n+1)
\rightarrow \cdots \rightarrow R \rightarrow R/(I_Z,F_1) \rightarrow 0.\]
denote this resolution; we shall also require this resolution
at the induction step.

Now suppose that $1 < r \leq \nu$, and assume by
induction that we have shown that $\mathbb{F}_{n-1} = \mathbb{F}'_{n-1} \oplus R(-d_1-n) \oplus \cdots
\oplus R(-d_{r-1}-n)$, and that a resolution of $(I_Z,F_1,\ldots,F_{r-1})$
is given by
\[ \mathcal{H}_{r-1}:
0 \rightarrow
\mathbb{F}'_{n-1} \oplus R^{\binom{n}{n-1}}(-d_1-n+1) \oplus \cdots \oplus R^{\binom{n}{n-1}}(-d_{r-1}-n+1)
\rightarrow \]
\[
\cdots \rightarrow R \rightarrow R/(I_Z,F_1,\ldots,F_{r-1}) \rightarrow 0.\]

We have a short exact sequence
\begin{equation}\label{ses2}
0\rightarrow R/((I_{Z},F_1,\ldots,F_{r-1}):(F_{r}))(-d_{r})
\xrightarrow{\times F_{r}} R/(I_{Z},F_1,\ldots,F_{r-1})
\end{equation}
\[\rightarrow R/(I_{Z},F_{1},\ldots,F_r)\rightarrow
0.\]
By Lemma \ref{idealepunto}, $R/((I_{Z},F_1,\ldots,F_{r-1}):(F_{r}))(-d_{r}) \cong
R/I_{P_i}(-d_r)$, so its resolution is given by
\[ \mathcal{K}_r: 0 \rightarrow R(-d_r-n) \rightarrow R^{\binom{n}{n-1}}(-d_r-n+1) \rightarrow \cdots \rightarrow R(-d_r)
\rightarrow  R/I_{P_i}(-d_{r}) \rightarrow 0.\]

Using the resolutions $\mathcal{K}_r$ and $\mathcal{H}_{r-1}$,
the short exact sequence ($\ref{ses2}$), and the mapping cone construction,
we have a
resolution of $R/(I_Z,F_1,\ldots,F_r)$ of the following form:
\[
0 \rightarrow R(-d_r-n) \rightarrow \mathbb{F}'_{n-1}
 \oplus R^{\binom{n}{n-1}}(-d_1-n+1) \oplus \cdots \oplus R^{\binom{n}{n-1}}
(d_{r}-n+1)
\rightarrow \]
\[
\cdots \rightarrow R \rightarrow R/(I_Z,F_1,\ldots,F_{r}) \rightarrow 0.\]

By Lemma \ref{idealepunto}, the ideal $(I_Z,F_1,\ldots,F_{r-1})$ is saturated,
so the ideal can have projective dimension at most $n-1$.  In other words,
the above resolution, which has length $n$, is too long.  As argued
above, $R(-d_r-n)$ must be part of the trivial complex and cancel
with some term in
\[ \mathbb{F}'_{n-1}
 \oplus R^{\binom{n}{n-1}}(-d_1-n+1) \oplus \cdots \oplus R^{\binom{n}{n-1}}
(-d_{r}-n+1).\]
Recall that the
definition of $\deg_Z(P_i) = (d_1,\ldots,d_{\nu})$ implies that
$d_1 \leq \cdots \leq d_r \leq \cdots \leq d_{\nu}$.
So $d_r + n > d_j + n -1$ for all $j=1,\ldots,r$.  Thus,
$R(-d_r-n)$ must cancel with some term in $\mathbb{F}'_{n-1}$,
i.e., $\mathbb{F}'_{n-1} = \mathbb{F}''_{n-1}\oplus R(-d_r-n)$.
Thus, $\mathbb{F}_n =  \mathbb{F}''_{n-1} \oplus R(-d_1-n) \oplus \cdots
\oplus R(-d_r-n)$,
and
\[ \mathcal{H}_{r}:
0 \rightarrow
\mathbb{F}''_{n-1} \oplus R^{\binom{n}{n-1}}(-d_1-n+1) \oplus \cdots \oplus R^{\binom{n}{n-1}}(-d_{r}-n+1)
\rightarrow \]
\[
\cdots \rightarrow R \rightarrow R/(I_Z,F_1,\ldots,F_{r}) \rightarrow 0\]
is a resolution of $R/(I_Z,F_1,\ldots,F_r)$.
This now completes the induction step.
\end{proof}

\begin{remark}
When all the $m_i$'s are one, that is, $Z$ is a set of reduced
points, our result recovers the results of Abrescia, Bazzotti,
and Marino \cite{ABM,Ba1,Ma3}.
\end{remark}

\begin{definition}\label{SB3}
Let $Z$ be a scheme of fat points  with  minimal graded free
resolution of type (\ref{minres}) with
$\mathbb{F}_{n-1}=\bigoplus_{j\in B_{n-1}}
R(-j)^{\beta_{(n-1),j}}$.
If $B_{n-1}= \{j_1,\ldots,j_t\}$, then
associate to $\mathbb{F}_{n-1}$ the vector
\[\mathcal{B}_{n-1} = (\underbrace{j_1,\ldots,j_1}_{\beta_{n-1,j_1}},\ldots,\underbrace{j_t,\ldots,j_t}_{\beta_{n-1,j_t}}).\]
For each integer $\tau \geq 1$, let
\[S_{\mathcal{B}_Z,\tau} = \{(j_{i_1}-n,\ldots,j_{i_\tau}-n) ~|~  j_{i_1} \leq \cdots \leq j_{i_\tau} ~\text{and}~
j_{i_1},\ldots,j_{i_{\tau}} \in \mathcal{B}_{n-1}\},\]
that is, the set of $\tau$-tuples whose entries are non-decreasing
and appear among the shifts of $\mathbb{F}_{n-1}$.
We call $S_{\mathcal{B}_Z,\tau}$ the {\it socle-vectors} of length
$\tau$ associated to the Betti numbers of $Z$.
\end{definition}

An example of the set of socle-vectors can be found in the
example following the next theorem.  Our next theorem shows that
the set of socle-vectors is a subset of the set of permissible
vectors.

\begin{theorem} Let $Z,Z'$ be as in Convention \ref{convention},
let $\nu = \deg Z - \deg Z'$
and let $H_Z$ be the Hilbert function of $Z$.
Then
\[\deg_Z(P_i) \in S_{\mathcal{B}_Z,\nu} \subseteq S_{H_Z,\nu}.\]
\end{theorem}

\begin{proof}
By Theorem \ref{teofat}, $\deg_Z(P_i) \in S_{\mathcal{B}_Z,\nu}$.

For each $\nu$-tuple $\underline{d} \in S_{\mathcal{B}_Z,\nu}$, we want
to show that $\underline{d} \in S_{H_Z,\nu}$.  Note that to show
that $\underline{d}$ is a permissible vector of length $\nu$
of $H_Z$, it suffices to show that the sequence
\[\{\Delta H_Z(t) - |\{d_i \in (d_1,\ldots,d_\nu) ~|~ d_i = t|\}\]
is the Hilbert function of an artinian quotient of $k[x_1,\ldots,x_n]$.
Then, by ``integrating'' this sequence, we obtain the sequence
\[\{H_Z(t) -  |\{d_i \in (d_1,\ldots,d_\nu) ~|~ d_i \leq t|\},\]
which will be the Hilbert function of a zero-dimensional scheme.

Let $S = R/I_Z$ be the coordinate ring of $Z$.  Since $S$ is a Cohen-Macaulay
ring of dimension one, we can pass to an artinian reduction $S'$ of
$S$.  That is, there exists a nonzero divisor $L$ of degree one such that
$S' \cong R/(I_Z,L)$ is an artinian ring.  Furthermore, since
$R/(I_Z,L) \cong (R/(L))/((I_Z,L)/(L))$, we can assume that
$S'$ is an artinian quotient of $k[x_1,\ldots,x_n]$.  So
$S' = k[x_1,\ldots,x_n]/J$ for some ideal $J$.

Because $S'$ is artinian, we can rewrite $S'$ as
\[S' = k \oplus S'_1 \oplus \cdots \oplus S'_t ~~\mbox{with $S'_t \neq 0$}\]
where $S'_i$ is the set of homogeneous elements of $S'$ of degree $i$.
The maximal ideal of $S'$ is then $m = \bigoplus_{i=1}^t S'_i$.
The socle of $S'$, denoted $\operatorname{soc}(S')$, is the annihilator of $m$.  In particular,
$\operatorname{soc}(S')$ is a homogeneous
ideal which we can write as the direct sum of its graded pieces:
$\operatorname{soc}(S') = T_1 \oplus \cdots \oplus T_t$
where $T_t = S'_t$.  The dimension of each $T_i$ is then
related to the graded Betti numbers of $J$.  In particular,
\[\dim_k T_i = \beta^{R'}_{n-1,n+i}(J) ~~\mbox{where $R' = k[x_1,\ldots,x_n]$.}\]
For more information about the socle and for proofs of these facts,
see \cite{GHMS}.

But because we are passing to an artinian reduction, the graded Betti numbers
of $I_Z$ and $J$ are the same, that is,
\[\beta^R_{n-1,n+i}(I_Z) = \beta^{R'}_{n-1,n+i}(J) ~~\mbox{for all $i \in \N$}.\]
Thus, if $\underline{d} = (d_1,\ldots,d_{\nu}) \in S_{\mathcal{B}_Z,\nu}$,
we can pick an element $G_i \in \operatorname{soc}(S')$ such that $\deg G_i = d_i$.  Moreover,
if $d_i = d_{i+1} = \cdots = d_{i+b}$, we can pick elements $G_i, \ldots,G_{i+b}$ that are linearly
independent socle elements since
$b \leq \beta^R_{n-1,n+d_i}(I_Z) = \beta^{R'}_{n-1,n+d_i}(J) = \dim_k T_{d_i}$.  That is, we take
$G_i,\ldots,G_{i+b}$ to be basis elements of $T_{d_i}$.

Thus, to $\underline{d} = (d_1,\ldots,d_{\nu})$ we can associate $\nu$ socle elements $\{G_1,\ldots,G_{\nu}\}$ of $S'$
such that $\deg G_i = d_i$, and if any subset of elements has the same degree,
then these elements are linearly independent over $k$.

We now want to compute the Hilbert function of $S'/(G_1,\ldots,G_{\nu})$.
We claim that for all $t \in \N$,
\[\dim_k(G_1,\ldots,G_{\nu})_t = |\{G_i \in \{G_1,\ldots,G_{\nu}\} ~|~ \deg G_i = t\}|.\]
We partition the elements of $\{G_1,\ldots,G_{\nu}\}$ into three sets, some
of which may be empty:
\begin{eqnarray*}
\mathcal{G}_< &= &\{G_1,\ldots,G_a\} = \{G_i \in \{G_1,\ldots,G_{\nu}\} ~|~ \deg G_i < t\}\\
\mathcal{G}_t &= &\{G_{a+1},\ldots,G_b\} = \{G_i \in \{G_1,\ldots,G_{\nu}\} ~|~ \deg G_i = t\}\\
\mathcal{G}_> &= &\{G_{b+1},\ldots,G_{\nu}\} = \{G_i \in \{G_1,\ldots,G_{\nu}\} ~|~ \deg G_i > t\}.
\end{eqnarray*}
By our choice of the $G_i$'s, the elements of $\mathcal{G}_t$ are linearly
independent, so $\dim_k(G_1,\ldots,G_{\nu})_t \geq |\mathcal{G}_t|$.
Now let $F$ be any element of $(G_1,\ldots,G_{\nu})_t$.  By degree
considerations, the elements of $\mathcal{G}_>$ do not contribute
to $(G_1,\ldots,G_{\nu})_t$.  So,
\[F = G_1A_1 + \cdots + G_aA_a + c_{a+1}G_{a+1} + \cdots +c_bG_b\]
where $G_1,\ldots,G_{a} \in \mathcal{G}_<$, $G_{a+1},\ldots,G_b \in \mathcal{G}_t$, $A_1,\ldots,A_a \in S'$, and
$c_{a+1},\ldots,c_b \in k$.  But since $\deg F = t$, $\deg A_i = t - \deg G_i > 0$
for $i=1,\ldots,a$.  This means that each $A_i$ is in $m$, which implies that
$G_iA_i = 0$ since each $G_i$ is a socle element.
Hence
\[F = c_{a+1}G_{a+1} + \cdots + c_bG_b\]
So, $F$ is in the vector space spanned by $\{G_{a+1},\ldots,G_b\}$, whence
$\dim_k (G_1,\ldots,G_{\nu})_t \leq |\mathcal{G}_t|$.
We have thus shown that for all $t \in \N$
\begin{eqnarray*}
H_{S'/(G_1,\ldots,G_{\nu})}(t) & = & H_{S'}(t) - |\{G_i \in \{G_1,\ldots,G_{\nu}\} ~|~ \deg G_i = t\}| \\
& = & H_{S'}(t) - |\{d_i \in (d_1,\ldots,d_{\nu}) ~|~ d_i = t\}|.
\end{eqnarray*}
Because $H_{S'}(t) = \Delta H_Z(t)$ for all $t$, this now completes the proof
since $H_{S'/(G_1,\ldots,G_{\nu})}$ is the Hilbert function of an artinian quotient
of $k[x_1,\ldots,x_n]$.
\end{proof}

\begin{example}
In $\pr^2$ let us consider two totally reducible forms $F$ and $G$
such that $\deg F = 3$ and $\deg G=7$, i.e., $F=L_{1}L_2L_{3}$
and $G=L'_{1}\cdots L'_{7}$ where the $L_i$s and $L'_i$s are linear
forms. Let $X=CI(3,7)$ be the complete
intersection of type $(3,7)$ defined by
$I_X = (F,G)$.  The 21 points of $X$
are the 21 points of intersection of the
$L_i$s and $L'_i$s, i.e., $P_{ij}=L_{i}\cap L'_{j}$ for
$i=1,2,3$ and $j=1,\ldots,7$. Set $Y=CI(3,7)\setminus \{P_{37}\}$ and
let $Z$ be the scheme of double points whose support is
the 20 points of $Y$, i.e.,
\[Z = 2P_{11} + \cdots + 2P_{36}.\]
We will now find the minimal separating set $\sf DEG_Z(2P_{36})$.
We let $Z_2 = Z$,  and
\[Z_1 = 2P_1 + \cdots + 2P_{35} + P_{36} ~\text{and}~ Z_0 = 2P_1 + \cdots + 2P_{35}.\]

By results found in \cite{GV1,GV2},
the minimal graded free resolution of $I_{Z_2}$ is:
\[0 \rightarrow R^2(-12)\oplus R(-15)\oplus R(-16) \rightarrow R(-6)\oplus R(-10)
\oplus R(-11)\oplus R^2(-14) \rightarrow I_{Z_2}\rightarrow 0.\]
Furthermore, the Hilbert function of $R/I_{Z_2}$ is
{\f \begin{center}
\begin{tabular}{ccccccccccccccccccccc}
{\f $t$} & : & \f{0} & \f{1} & \f{2} & \f{3} & \f{4} & \f{5}&
\f{6}&\f{7} & \f{8} & \f{9} & \f{10} & \f{11} & \f{12} & \f{13} &
 \f{14} & \f{15}&\f{16}\\
$H_{Z_2}(t)$ & :  &1 &3 &6 &10 &15 &21&27&33&39  &45 &50
&53&56&59&60&$\rightarrow$
&$\rightarrow$\\
$\Delta H_{Z_2}(t)$& :  &1 &2 &3 &4 &5 &6&6 &6&6&6 &5 &3
&3&3&1&0&$\rightarrow$\\
\end{tabular}
\end{center}}
\noindent
By Theorem \ref{teofat},
the degree of the minimal separators of $P_{36}$ of multiplicity 2
must be one of the elements of $S_{{B_Z},2}$.  From the resolution
of $I_Z$, we compute the vector
\[\mathcal{B}_{2-1} = (12,12,15,16).\]
Then the set of socle vectors of length 2 is
\[
S_{B_Z,2}= \{(10,10),(10,13),(10,14),(13,14)\},\]
i.e., $\deg_Z(P_{36})$ is one of these four tuples.
We use CoCoA \cite{C} to compute the minimal graded free resolution
of $I_{Z_1}$:
\[ 0 \rightarrow R(-11)\oplus R(-12)\oplus R(-14)\oplus R(-16)
\rightarrow R(-6)\oplus R^2(-10)\oplus R(-13)\oplus R(-14)
\rightarrow I_{Z_1}\rightarrow 0\]
\noindent
and its first difference Hilbert function is
\[\begin{tabular}{cccccccccccccccccccc}
{\f $t$} & : & \f{0} & \f{1} & \f{2} & \f{3} & \f{4} & \f{5}&
\f{6} &\f{7} & \f{8} & \f{9} & \f{10} & \f{11} & \f{12} & \f{13} &
 \f{14} & \f{15}&\f{16}\\
$\Delta H_{Z_1}(t)$& :  &1 &2 &3 &4 &5 &6&6 &6&6&6 &4 &3
&3&2&1&0&$\rightarrow$\\
\end{tabular}\]
\noindent
By comparing $\Delta H_{Z_1}$ and $\Delta H_{Z_2}$,
Theorem \ref{hilbertfunction} reveals that
$\deg_{Z_2}(P_{36}) = (10,13)$.  Furthermore,
by Theorem \ref{teofat}, $\deg_{Z_1}(P_{36})$ must be
one of $\{(11-2),(12-2),(14-2),(16-2)\} = \{(9),(10),(12),(14)\}$.

If we compute the Hilbert
function of  $R/I_{Z_{0}}$
we get
\[\begin{tabular}{cccccccccccccccccccc}
{\f $t$} & : & \f{0} & \f{1} & \f{2} & \f{3} & \f{4} & \f{5}&
\f{6} & \f{7} & \f{8} & \f{9} & \f{10} & \f{11} & \f{12} & \f{13} &
 \f{14} & \f{15}&\f{16}\\
$\Delta H_{Z_0}(t)$& :  &1 &2 &3 &4 &5 &6&6& 6 & 6&6 &4 &3
&2&2&1&0 &$\rightarrow$\\
\end{tabular}\]
which reveals that $\deg_{Z_1}(P_{36})=(12)$.

Thus, the minimal separating set of the fat point $2P_{36}$ is
the set ${\sf
{DEG}}_Z(2P_{36})=\{(12),(10,13)\}$.
\end{example}

As an interesting corollary to Theorem \ref{teofat}, we get a bound on the rank of the last syzygy
module in terms of the $m_i$s and $n$.

\begin{corollary}\label{rank}
Let $Z = m_1P_1 + \cdots + m_sP_s \subseteq \pr^n$ be
a set of fat points, and let $m = \max\{m_1,\ldots,m_s\}$.
If
\[0 \rightarrow \mathbb{F}_{n-1} \rightarrow \cdots \rightarrow
\mathbb{F}_0 \rightarrow I_Z \rightarrow 0 \] is a minimal graded
free resolution of $I_Z$, then
\[\operatorname{rk}\mathbb{F}_{n-1} \geq \binom{m+n-2}{n-1}.\]
\end{corollary}

\begin{proof}
Suppose $P_i$ has multiplicity $m$.  Then by Theorem \ref{teofat}, the
syzygy module $\mathbb{F}_{n-1}$ must have at least $\nu = \deg Z - \deg Z'
= \binom{m+n-2}{n-1}$ shifts.  The conclusion now follows.
\end{proof}


\section{Application: a Cayley-Bacharach type of result}

We use Theorem \ref{teofat} to produce a
Cayley-Bacharach (to be defined below) type of result for homogeneous
sets of fat points
in $\pr^n$ whose support is a complete intersection (see
\cite{GV1,GV2} and references there within, for more on these special configurations).
In particular, we show that if $Z$
is a homogeneous fat point scheme whose support
is a complete intersection, then $\deg_Z(P)$ is the
same for every point $P \in \operatorname{Supp}(Z)$.  We prove
this result by showing that the last syzygy module of $I_Z$
only permits one possible choice for $\deg_Z(P)$.
We also show how to calculate $\deg_Z(P)$ in this situation.

Let $X \subseteq \pr^n$ be a complete intersection of points of
type $(\delta_1,\ldots,\delta_n)$.  This means that $I_X =
(F_1,\ldots,F_n)$ where $F_1,\ldots,F_n$ define a complete
intersection with $\deg F_i = \delta_i$ for all $i=1,\ldots,n$.
Without loss of generality, we can assume that $\delta_1 \leq
\cdots \leq \delta_n$. We now recall a result which is a special
case of a classical result of Zariski-Samuel \cite[Lemma 5,
Appendix 6]{ZS}.

\begin{lemma}
Suppose that $X = \{P_1,\ldots,P_s\} \subseteq \pr^n$ is a complete
intersection of reduced points.  For any integer $m > 1$, the defining ideal of the
homogeneous fat point scheme $Z = mP_1 + \cdots + mP_s$
is given by
$I_Z = I_X^m.$
\end{lemma}

The ideal of $Z$ is then a power of a complete intersection. In
\cite{GV1}, the first and third authors described the graded Betti
numbers in the graded minimal free resolution of the power of
any complete intersection in terms of the type.  As a special case,
we can describe all the shifts at the end of the
resolution of $I_Z$.

\begin{theorem}     \label{maintheorem}
Suppose that $X = \{P_1,\ldots,P_s\} \subseteq \pr^n$ is a complete
intersection of reduced points of type $(\delta_1,\ldots,\delta_n)$.
For any integer $m > 1$,  the minimal graded free resolution of the
ideal $I_Z$ defining the
homogeneous fat point scheme $Z = mP_1 + \cdots + mP_s$
has the form
\[0 \rightarrow \mathbb{F}_{n-1} \rightarrow \cdots \rightarrow \mathbb{F}_0 \rightarrow
I_Z = I_X^m \rightarrow 0\]
where
\[\mathbb{F}_{n-1}  =  \bigoplus_{(a_1,\ldots,a_n) \in \M_{n,m+n-1}}
R(-a_1\dd_1-\cdots-a_n \dd_n).\]
Here, the set
\[\mathcal{M}_{n,m+n-1} := \left\{ (a_1,\ldots,a_n)\in \N^n
\left|\begin{array}{c}
a_1 + \cdots + a_n = m+n-1 ~~\mbox{and }\\
\mbox{$a_i \geq 1$ for all $i$}
\end{array} \right\}\right..\]
\end{theorem}

\begin{proof}
See Theorem 2.1. in \cite{GV1}.
\end{proof}

\begin{corollary}\label{rkcor}
With the hypotheses as in Theorem \ref{maintheorem},
\[\operatorname{rk} \mathbb{F}_{n-1}  = \binom{m+n-2}{n-1}.\]
\end{corollary}

\begin{proof}
By Theorem \ref{maintheorem}, the set of
integer solutions to
$a_1 + \cdots + a_n = m+n-1$ with all $a_i \geq 1$
is in bijection with the generators of the free module
$\mathbb{F}_{n-1}$.  The number of integer solutions
to this equation is $\binom{m+n-2}{n-1}$.
\end{proof}

Every fat point in a homogeneous fat point scheme
whose support is a complete intersection must now
have the same degree:

\begin{theorem}\label{degCI}
Let $Z = mP_1 + \cdots + mP_s \subseteq \pr^n$ be a homogeneous
fat point scheme such that $\supp(Z)$ is a complete intersection.
Then, for every $P_i \in \supp(Z)$,  the tuple $\deg_Z(P_i)$ is
the same.   In particular, for every $P_i \in Z$,
the schemes $Z' = mP_1 + \cdots + (m-1)P_i + \cdots
+ mP_s$ all have the same Hilbert function.
\end{theorem}

\begin{proof}
By Theorem \ref{teofat}, each of the $\nu = \deg Z - \deg Z'$ entries
of $\deg_Z(P_i) = (d_1,\ldots,d_{\nu})$ appear as shifts
of the form $-d_i-n$ among
the shifts of the $(n-1)$-th syzygy module of $I_Z$.  But by Corollary
\ref{rkcor}, there are exactly $\nu$ such shifts in $\mathbb{F}_{n-1}$
when $Z$ is a homogeneous fat point scheme whose support is
a complete intersection.  Thus, for each $P_i \in \supp(Z)$,
there is only choice for $\deg_Z(P_i)$.
\end{proof}

The above result can be interpreted as saying that
homogeneous fat point schemes whose support is a complete intersection
have a property similar to the Cayley-Bacharach property for
reduced points.
We recall this definition:

\begin{definition}
A set of reduced points $X = \{P_1,\ldots,P_s\} \subseteq \pr^n$ is said
to have the {\it Cayley-Bacharach property} (CBP) if for every $P
\in X$, the Hilbert function $H_{X\setminus \{P\}}$ is the same.
\end{definition}

Using Corollary \ref{hilbertonept}, one can prove:

\begin{theorem}  \label{CBPclassify}
Let $X = \{P_1,\ldots,P_s\}$
be a set of reduced points in $\pr^n$.
Then $X$ has the CBP if and only if $\deg_X(P_1) = \cdots = \deg_X(P_s)$.
\end{theorem}

In Theorem \ref{degCI}, we showed that $\deg_Z(P_1) =
\cdots = \deg_Z(P_s)$ when $\supp(Z) = \{P_1,\ldots,P_s\}$ is
a complete intersection.   By analogy with Theorem \ref{CBPclassify},
this suggests that homogeneous fat point schemes whose
support is a complete intersection have a property similar
to the reduced sets of points with the CBP.
There are many examples of reduced sets
points with the CBP:  level sets of points,
Gorenstein sets of points, and complete intersections (the last
two are examples of the first).  It would be interesting to
find other classes of fat point schemes $Z$
which have the property that $\deg_Z(P_1) =
\cdots = \deg_Z(P_s)$ when $\supp(Z) = \{P_1,\ldots,P_s\}$.

\begin{remark}
In her Ph.D. thesis \cite{Gu1}, the first author introduced the definition
of a Cayley-Bacharach property for homogeneous schemes of
fat points in $\pr^2$ whose support is a complete intersection of
type $(a,b)$:

\begin{definition} \label{cb}
In $\pr^2$, a homogeneous scheme of fat points $Z$ whose
support is a complete intersection of type $(a,b)$ has
the {\it Cayley-Bacharach property} if for all $i = 1,\ldots,ab$,
the subschemes of $Z$ of type
\[Y_i = mP_1 +\cdots + \widehat{mP_i} + \cdots + mP_{ab} ~\mbox{with}~~ \deg(Y)=\deg(X)-
\binom{m+1}{2},\] have the same Hilbert function.
\end{definition}

Theorem 3.5.4 and Corollary 3.5.5 in \cite{Gu1} showed
that when $m=2$, all the homogeneous schemes of double points with
support a complete intersection have the Cayley-Bacharach property.
Note that the point-of-view taken in this definition is different
from the one we have used in this paper.  The schemes
being studied in \cite{Gu1} have ``removed'' the entire fat point, while
in this paper we have focused on what happens when we ``reduce''
the multiplicity of a point.
\end{remark}

Using Theorems \ref{maintheorem} and \ref{degCI}
we can actually calculate $\deg_Z(P) = (d_1,\ldots,d_{\nu})$
when $Z$ is a homogeneous fat point scheme supported on
a complete intersection directly from the
type of the complete intersection.  We illustrate this behavior
via an example.

\begin{example}
Consider a complete intersection of points $X$ in $\pr^3$ of type
$(2,3,4)$, and consider the homogeneous scheme of fat points $Z$
of multiplicity $m =3$ supported on $X$. Then
\[\mathcal{M}_{3,3+3-1} := \left\{ (a_1,a_2,a_3)\in \N^3
\left|\begin{array}{c}
a_1 + a_2+a_3 = 3+3-1 =5 ~~\mbox{and }\\
\mbox{$a_i \geq 1$ for all $i$}
\end{array} \right\}\right..\]
This set only contains six elements: \[\mathcal{M}_{3,5} =
\{(1,1,3),(1,3,1),(1,1,3),(1,2,2),(2,1,2),(2,2,1)\}\,.\] Thus, by
Theorem \ref{maintheorem}, the last syzygy module $\mathbb{F}_2$
in the resolution of $I_Z = I_X^3$ has the form:
\[R(-1\cdot2 - 1\cdot3- 3\cdot 4)\oplus R(-1\cdot2 - 3\cdot3- 1\cdot 4)
\oplus
R(-3\cdot2 - 1\cdot3- 1\cdot
4)\] \[\oplus R(-1\cdot2 -
2\cdot3- 2\cdot 4)\oplus R(-2\cdot2 - 1\cdot3- 2\cdot 4)\oplus
R(-2\cdot2 - 2\cdot3- 1\cdot 4)\]
\[ = R(-13)\oplus R(-14)\oplus R^2(-15) \oplus R(-16) \oplus R(-17).\]
Thus, for any $P \in \supp(Z)$, Theorems \ref{teofat} and \ref{degCI}
give
\[\deg_Z(P) = (13-3,14-3,15-3,15-3,16-3,17-3) = (10,11,12,12,13,14).\]
\end{example}


\end{document}